\documentclass[preprint]{elsarticle}
\usepackage{amsmath}
\usepackage{amssymb,amsfonts}
\usepackage{amsthm}
\usepackage{tikz}
\usepackage{fullpage}
\usepackage{graphicx}
\usepackage{todonotes}
\usepackage[bookmarks,colorlinks,breaklinks]{hyperref}  
\hypersetup{linkcolor=blue,citecolor=blue,filecolor=blue,urlcolor=blue} 
\numberwithin{equation}{section}
\newcommand{\lref}[2][]{\hyperref[#2]{#1~\ref*{#2}}}
\renewcommand{\eqref}[2][]{\hyperref[#2]{(\ref*{#2})}}
\usepackage{algorithm}
\usepackage{algorithmic} 

\newtheorem{theorem}{Theorem}

\newtheorem{lemma}[theorem]{Lemma}

\theoremstyle{definition}

\newtheorem{reduction}{Reduction Rule}

\def\beginrefs{\begin{list}%
		{[\arabic{equation}]}{\usecounter{equation}
			\setlength{\leftmargin}{2.0truecm}\setlength{\labelsep}{0.4truecm}%
			\setlength{\labelwidth}{1.6truecm}}}
	\def\endrefs{\end{list}}

\newcommand{\size}[1]{\left| #1 \right|}

\newcommand{\cop}[1]{$\mathcal{C}_{#1}$}
\newcommand{\robber}{$\mathcal{R}$}

\begin{document}

\title{A bound for the cops and robber problem in terms of 2-component order connectivity}

\author{Suryaansh Jain}
\ead{cs21btech11057@iith.ac.in}
\author{Subrahmanyam Kalyanasundaram}
\ead{subruk@cse.iith.ac.in}

\author{Kartheek Sriram Tammana}
\ead{cs21btech11028@iith.ac.in}

\affiliation
{organization={Department of Computer Science and Engineering},
            addressline={IIT Hyderabad}, 
            city={Hyderabad},
            postcode={502285}, 
            country={India}}

\begin{abstract}
    In the cops and robber game, there are multiple cops and a single robber taking turns moving along the edges of a graph. The goal of the cops is to capture the robber (move to the same vertex as the robber) and the goal of the robber is to avoid capture. 
    The \emph{cop number} of a given graph is the smallest number of cops 
    required to ensure the capture of the robber. 
    The $\ell$-\emph{component order connectivity} of a graph $G = (V, E)$ is the size of a smallest set $U$, such that all the connected components of the induced graph on $V\setminus U$ are of size at most $\ell$. In this brief note, we provide a bound on the cop number of graphs in terms of their 2-component order connectivity. 
\end{abstract}
\maketitle

\section{Introduction}


The game of \emph{cops and robber} was introduced in 1984 by Aigner and Fromme \cite{aignerfromme}. The game is played on an undirected graph. There could be multiple cops and one robber. 
The game starts out with the cops placing themselves on the vertices of the graph, followed by the robber. The cops and the robber now take turns to move, starting with the cops. A move is defined as going to any neighbour, or staying put at the current vertex. The goal of the cops, who play as a team, is to \emph{capture} the robber, which means to move to the same vertex as the robber.
The goal of the robber is to evade capture. The game ends when the robber is caught. The cop number of a graph $G$, denoted by $c(G)$, is the minimum number of cops required such that the robber cannot evade the cops indefinitely.

The cops and robber game has been studied extensively, see for instance \cite{scott-sudakov}. The Meyniel's conjecture is a famous open question that states that for any graph $G$, the cop number $c(G) = O(\sqrt{n})$ \cite{meynielsurvey}. The problem has been studied from different perspectives: on various classes of  graphs \cite{P5-free, forbidden}, characterizations \cite{k-copwin}, and variants of the original problem \cite{ride-hide, surround}.

Recently, the problem has been investigated from the perspective of parameterized algorithms, and there are bounds on the cop number with respect to structural parameters of the graph such as vertex cover number \cite{gahlawat} and genus of the graph \cite{genusbound}.
The paper \cite{gahlawat} bounds the cop number as $c(G) \leq \frac{\mbox{vcn}(G)}{3} + 1$, where  $\mbox{vcn}(G)$ is the vertex cover number, the size of a smallest vertex cover of $G$.

In this note, we consider a generalization of vertex cover, called \emph{component order connectivity} \cite{grosssurveycoc}. The $\ell$-component order connectivity (or $\ell$-coc) of a graph $G = (V, E)$ is the size of a smallest set $U$, such that all the connected components of the induced graph on $V\setminus U$ are of size at most $\ell$. By setting $\ell =1$, we get the vertex cover number of $G$. We state our result below.

\begin{theorem}\label{thm:main}
    For an undirected graph $G$, the cop number $c(G) \leq \frac{\mbox{2-coc}(G)}{3} + 4$.
\end{theorem}

\subsection{Basic Notation and Preliminaries} 
We consider simple undirected graphs. We use standard graph theoretic terminology from Diestel \cite{Diestel}. The \emph{open neighborhood} of a vertex $v$ of $G$ is the set of neighbors of $v$ and is denoted $N(v)$. The \emph{closed neighborhood} of $v$, denoted $N[v]$, is defined $N[v] = N(v) \cup \{v\}$.


\section{Proof of Theorem \ref{thm:main}}


Consider a graph $G = (V, E)$ with 2-coc cover $U$. This means that each connected component of $G[V\setminus U]$ is of size at most 2.  We first apply  the following reduction rules on graph $G$. These rules involve placing cops at certain vertices. These cops will limit the potential vertices to which the robber can move to. Note that these are the same rules from \cite{gahlawat}. We use \robber{} to denote the robber, and \cop{i} to denote the $i$-th cop.

\begin{reduction}[RR1]\label{rr:rr1}
    If there is a vertex $v \not\in U$ such that $|N(v) \cap U| \geq 3$, then place a cop at $v$ and delete $N[v]$.
\end{reduction}

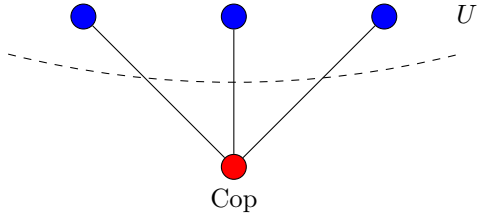
\begin{center}
\begin{tikzpicture}
    \node[draw,circle,fill=blue] (A) at (1,2) {};
    \node[draw,circle,fill=blue] (B) at (3,2) {};
    \node[draw,circle,fill=blue] (C) at (5,2) {};
    \node[draw,circle,fill=red,label=below: Cop] (D) at (3,0) {};
    \node[label=below: $U$] (E) at (6.1,2.4) {};
    \draw (B) -- (D);
    \draw (C) -- (D);
    \draw (A) -- (D);
    \draw[dashed] (0,1.5) .. controls (2,1) and (4,1) .. (6,1.5);
\end{tikzpicture}
\end{center}

\begin{reduction}[RR2]\label{rr:rr2}
    If there is a vertex $v \in U$ such that $\size{N(v) \cap U} \geq 2$, then place a cop at $v$ and delete $N[v]$.
\end{reduction}

\begin{center}
\begin{tikzpicture}
    \node[draw,circle,fill=blue] (A) at (1,2) {};
    \node[draw,circle,fill=blue] (C) at (5,2) {};
    \node[draw,circle,fill=red,label=above: Cop] (D) at (3,3) {};
    \node[label=below: $U$] (E) at (6.1,2.4) {};
    \draw (C) -- (D);
    \draw (A) -- (D);
    \draw[dashed] (0,1.5) .. controls (2,1) and (4,1) .. (6,1.5);
\end{tikzpicture}
\end{center}

\begin{reduction}[RR3]\label{rr:rr3} \cite{aignerfromme}
    If there exists an isometric (shortest) path with at least 3 vertices in $U$ then use one cop to protect the path and delete it.
\end{reduction}

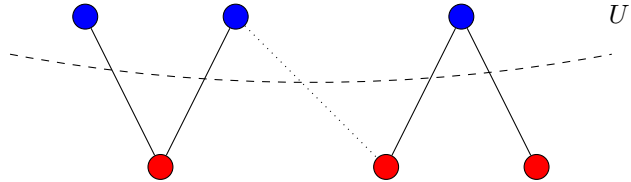
\begin{center}
\begin{tikzpicture}
    \node[draw,circle,fill=blue] (A) at (1,2) {};
    \node[draw,circle,fill=blue] (B) at (3,2) {};
    \node[draw,circle,fill=blue] (C) at (6,2) {};
    \node[draw,circle,fill=red] (X) at (2,0) {};
    \node[draw,circle,fill=red] (Y) at (5,0) {};
    \node[draw,circle,fill=red] (Z) at (7,0) {};
    \node[label=below: $U$] (E) at (8.1,2.4) {};
    \draw (A) -- (X);
    \draw (X) -- (B);
    \draw[dotted] (B) -- (Y);
    \draw (Y) -- (C);
    \draw (C) -- (Z);
    \draw[dashed] (0,1.5) .. controls (2.67,1) and (5.33,1) .. (8,1.5);
\end{tikzpicture}
\end{center}

Note that when we say a vertex is deleted during a reduction, it is not actually deleted, but is guarded by a cop, and therefore not accessible to \robber{}. For the cops the vertex is accessible and the graph remains connected.
Suppose after exhaustive application of the above reductions, the graph $G$ is reduced to $G'$ (from the perspective of the robber). Each reduction removes at least 3 vertices of $U$, and say $r$ reductions were applied in total.
Defining $U' = G' \cap U$, we now have

\begin{equation} \label{eq:bound1}
    |U'| \leq |U| - 3r.
\end{equation}

Since the $r$ cops we added protect every vertex in $G \setminus G'$, the robber is constrained to $G'$. In particular, the robber is constrained to a connected component of $G'$. 
So we need to find the maximum cop number of any connected component of $G'$. We have the following.

\begin{equation} \label{eq:bound2}
    c(G) \leq r + \max_{H \text{ is a connected component of } G'} c(H)\;.
\end{equation}

For any connected component $H$ of $G'$, define $U'_H = U' \cap H$. The following lemma now follows easily from \lref[RR]{rr:rr3}.

\begin{lemma}\label{ll:l1}
    Any two vertices in $U'_H$ have distance at most 3 in $H$. Any shortest path between the two vertices will not contain any vertices within $U$.
\end{lemma}

\begin{proof}
    If a shortest path were to contain a vertex in $U$ apart from the two end points, then \lref[RR]{rr:rr3} would have applied, but we have exhaustively applied all the reduction rules.
    A shortest path cannot have length more than $3$ because the size of components outside $U$ is at most $2$.

    Therefore, any two vertices in $U'_H$ have distance at most 3 in $H$, and this shortest path will not contain any vertices within $U$.
\end{proof}

\begin{lemma}\label{ll:l2}
    Let $H$ be a connected compoenent of $G'$. In the subgraph $H$, three cops are sufficient to (after a finite number of turns) force \robber{} to move in a path that does not use an edge with both vertices outside $U'_H$.
\end{lemma}

\begin{proof}
    We give the strategy for the three cops that forces the robber (referred to as \robber{}) into a path as described in the statement.
    Name the cops \cop{1}, \cop{2}, and \cop{3}. All the three cops start on arbitrary vertices in $U'_H$. We define \emph{backtrack} as a move in the reverse direction along the same edge used in the previous turn.

\iftrue
    As a consequence of Lemma \ref{ll:l1}, it follows that $H$ has diameter at most 7. This happens when we have a shortest path $v_0, v_1, \ldots, v_7$ in $H$ where $v_2$ and $v_5$ are the only vertices in $U'_H$.
    
    \cop{1} initially moves to the starting point of \robber{} (which is at most 7 vertices away), and then follows the path traced by \robber{}, except that \cop{1} continues to move forward along the path even when the robber stays put or backtracks.
    This implies that the robber can backtrack or stay put at most 7 times throughout the whole game, otherwise \cop{1} will catch up to him.
    
    The cops \cop{2} and \cop{3} stay put when \robber{} stays put, and backtrack when \robber{} backtracks (we define the rest of \cop{2} and \cop{3}'s strategy later on). In effect, the relative positions of \cop{2} and \cop{3} with respect to \robber{} do not change when \robber{} backtracks/stays put, while \cop{1} gets closer.
    This means that backtracking is always disadvantageous to \robber{}.
    Hence we now consider the case that \robber{} neither backtracks nor stays put.
\fi

    As \cop{1} chases \robber{}, consider a time that the robber moves out of $U'_H$. Until this point, \cop{2} and \cop{3} do not move. We will now explain how \cop{2} and \cop{3} respond to the robber.

\begin{figure}[h]
\begin{center}
\begin{tikzpicture}
    \node[draw,circle,fill=gray, label=right: $a$] (A) at (1,2) {};
    \node[draw,circle,fill=gray, label=right: $b$] (B) at (3,2) {};
    \node[draw,circle,fill=cyan, label=right: $d$] (D) at (5,2) {};
    \node[draw,circle,fill=blue, label=right: $e$] (E) at (7,2) {};
    \node[draw,circle,fill=blue, label=right: \cop{3}] (F) at (10,2) {};
    \node[draw,circle,fill=cyan, label=right: \cop{2}] (H) at (12,2) {};
    \node[draw,circle,fill=red, label=below right: $r$ (\robber{})] (X) at (2,0) {};
    \node[draw,circle,fill=gray, label=below: $y$] (Y) at (6,0) {};
    \node[draw,circle,fill=gray] (Z) at (8,0) {};
    \node[draw,circle,fill=gray] (W) at (9,0) {};
    \node[draw,circle,fill=gray] (P) at (7,-0.5) {};
    \node[draw,circle,fill=gray] (Q) at (10,-0.5) {};
    \node[label=below: $U'_H$] (G) at (14,2.5) {};
    \draw (A) -- (X);
    \draw (X) -- (B);
    \draw (X) -- (Y);
    \draw (Y) -- (D);
    \draw (Y) -- (E);
    \draw (W) -- (Z);
    \draw (W) -- (F);
    \draw (Z) -- (E);
    \draw (D) -- (P);
    \draw (P) -- (Q);
    \draw (Q) -- (H);
    \draw[dashed] (0,1.5) .. controls (4.67,1) and (9.33,1) .. (14,1.5);
\end{tikzpicture}

\caption{Proof of Lemma \ref{ll:l2}. The vertex $r$ is the first instance when \robber{} moves outside of $U'_H$. Assuming that $r$ is in a component of size 2 in $G'[V \setminus U]$, we let $y$ be the neighbor of $r$ outside $U'_H$. The vertices $a, b$ are the two neighbors of $r$ in $U'_H$. The vertices $d, e$ are the two neighbors of $y$ in $U'_H$.}
\label{fig:lem}
\end{center}
\end{figure}
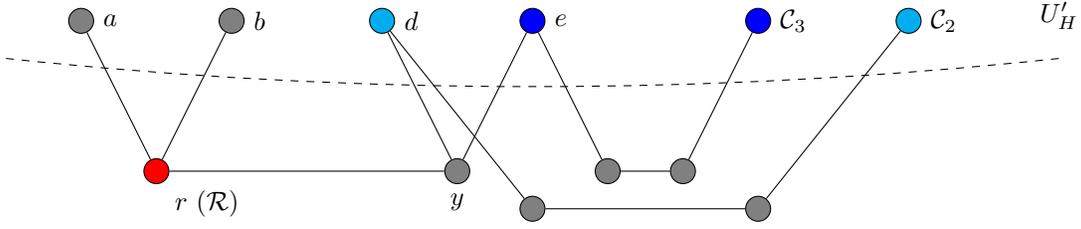

    Suppose the robber moves from vertex $a$ in $U'_H$ to a vertex $r$ outside $U'_H$, as shown in Figure \ref{fig:lem}. At this point, \cop{2} and \cop{3} are at some vertices in $U'_H$. The figure shows the ``largest case'', where $r$ is part of a component of size 2 outside $U'_H$ (call the other vertex $y$), and $r$ and $y$ have 2 neighbours each in $U'_H$. 
    All the other cases can be addressed by considering a subset of this case.
    Therefore the algorithm can be applied to these smaller cases where $r$ or $y$ has less than $2$ neighbours within $U'_H$.
    The vertex $b$ is the other neighbor of $r$ (apart from $a$) in $U'_H$. The vertices $d, e$ are the two neighbors of $y$ in $U'_H$.

    After \robber{} moves to $r$, \cop{2} moves along its shortest path towards $d$ and \cop{3} similarly moves towards $e$. By \lref[Lemma]{ll:l1}, we see that \robber{} cannot reach $d$ or $e$ without getting caught.
    
    This means \robber{} cannot move to $y$ on his next turn, since he cannot then backtrack to $r$ due to the presence of \cop{1}.
    Similarly, \robber{} cannot backtrack to $a$ either. This means that on the next move, \robber{} must move to $b$, which is in $U'_H$.
    Once \robber{} moves to $b$, \cop{2} and \cop{3} move back to initial previous positions, and the above argument repeats, ensuring that the robber never uses an edge completely out of $U'_H$.
\end{proof}

Now \robber{} is no longer allowed to use any edge that has both its end points in $H \setminus U'_H$. 
Let $\widehat{H}$  be the subgraph of $H$ after removing such edges.
This means that every edge of $\widehat{H}$ has at least one endpoint in $U'_H$, meaning $U'_H$ is a vertex cover of $\widehat{H}$. Using the bound that relates the cop number  to the vertex cover number from \cite{gahlawat}, we get that
\begin{equation}
c(\widehat{H}) \leq \frac{\mbox{vcn}({\widehat{H}})}{3} + 1 \leq \frac{\size{U'_H}}{3} + 1 \leq \frac{\size{U'}}{3} + 1 \;.
\end{equation}

Applying \lref[Lemma]{ll:l2}, we have
\begin{equation} \label{eq:bound3}
c(H) \leq 3 + c(\widehat{H}) \leq \frac{\size{U'}}{3} + 4\;.
\end{equation}

Combining \eqref{eq:bound1}, \eqref{eq:bound2}, and \eqref{eq:bound3}, we have
\begin{align}
    c(G) &\leq \frac{\size{U}}{3} + 4\,.
\end{align}

\section{Open Questions}
A clear direction is to find bounds that relate the cop number to the size of the smallest feedback vertex set.

Vertex cover number is 1-coc, and our result relates the cop number to 2-coc. A natural approach is to bound the cop number by $k$-coc for a constant $k$. 

Could we get bounds of the following form?
\begin{align*}
    c(G)  = \frac{|U|}{3} + f(k) \;.
\end{align*}
The function $f(k)$ could have an arbitrary dependence on $k$. 
It is also possible that the 3 in the denominator could be replaced by a smaller constant, say $2 + \epsilon$, for a constant $\epsilon >0$.


\section{Acknowledgments}

The second author wishes to acknowledge SERB-DST for supporting this work via the grant CRG/2022/009400.

\nocite{Diestel}

\bibliographystyle{alpha}

\bibliography{bibfile}
    
\end{document}